\documentclass[12pt]{amsart}
\usepackage{fullpage,url,amssymb}

\DeclareFontEncoding{OT2}{}{} 

\usepackage{color}

\newcommand{\defi}[1]{\textsf{#1}} 

\newcommand{\C}{{\mathbb C}}
\newcommand{\F}{{\mathbb F}}

\newcommand{\PP}{{\mathbb P}}
\newcommand{\Q}{{\mathbb Q}}

\DeclareMathOperator{\Char}{char}

\newtheorem{theorem}{Theorem}[section]
\newtheorem{lemma}[theorem]{Lemma}

\theoremstyle{definition}

\theoremstyle{remark}
\newtheorem{remark}[theorem]{Remark}

\usepackage[
	backref,
	pdfauthor={Bjorn Poonen}, 
]{hyperref}
\usepackage[alphabetic,backrefs,lite]{amsrefs} 

\begin{document}

\title[Curves violating the local-global principle]{Curves over every global field violating the local-global principle}
\subjclass[2000]{Primary 11G30; Secondary 14H25}
\keywords{Hasse principle, local-global principle, Dem'janenko-Manin method}
\author{Bjorn Poonen}
\thanks{This research was supported by NSF grant DMS-0841321.}
\address{Department of Mathematics, Massachusetts Institute of Technology, Cambridge, MA 02139-4307, USA}
\email{poonen@math.mit.edu}
\urladdr{http://math.mit.edu/~poonen}
\date{May 14, 2010}

\begin{abstract}
There is an algorithm that takes as input a global field $k$
and produces a curve over $k$
violating the local-global principle.
Also, given a global field $k$ and a nonnegative integer $n$,
one can effectively construct a curve $X$ over $k$ such that $\#X(k)=n$.
\end{abstract}

\maketitle

\section{Introduction}\label{S:introduction}

Let $k$ be a global field, by which we mean a finite extension
of either $\Q$ or $\F_p(t)$ for some prime $p$.
Let $\Omega_k$ be the set of nontrivial places of $k$.
For each $v \in \Omega_k$, let $k_v$ be the completion of $k$ at $v$.
By \defi{variety}, we mean a separated scheme of finite type over a field.
A \defi{curve} is a variety of dimension $1$.
Call a variety \defi{nice} if it is smooth, projective, and
geometrically integral.
Say that a $k$-variety $X$ satisfies the \defi{local-global principle}
if the implication
\[
	X(k_v) \ne \emptyset \text{ for all $v \in \Omega_k$}
	\implies X(k) \ne \emptyset
\]
holds.

Nice genus-$0$ curves (and more generally, quadrics in $\PP^n$)
satisfy the local-global principle:
this follows from the Hasse-Minkowski theorem for quadratic forms.
The first examples of varieties violating the local-global principle
were genus-$1$ curves, such as the smooth projective model of $2y^2=1-17x^4$,
over $\Q$, discovered by Lind \cite{Lind1940} 
and Reichardt \cite{Reichardt1942}.

Our goal is to prove that
there exist curves over every global field
violating the local-global principle.
We can also produce curves having a prescribed positive number of 
$k$-rational points.
In fact, such examples can be constructed effectively:

\begin{theorem}
\label{T:main}
There is an algorithm that takes as input a global field $k$
and a nonnegative integer $n$, 
and outputs a nice curve $X$ over $k$ 
such that $\#X(k)=n$
and $X(k_v) \ne \emptyset$ for all $v \in \Omega_k$.
\end{theorem}

\begin{remark}
For the sake of definiteness,
let us assume that $k$ is presented by giving the minimal polynomial
for a generator of $k$ as an extension of $\Q$ or $\F_p(t)$.
The output can be described by giving a finite list of homogeneous
polynomials that cut out $X$ in some $\PP^n$.
For more details on representation 
of number-theoretic and algebraic-geometric objects, 
see \cite{Lenstra1992}*{\S2} and \cite{Baker-et-al2005}*{\S5}.
\end{remark}

\section{Proof}

\begin{lemma}
\label{L:at least 1}
Given a global field $k$,
one can effectively construct a nice curve $Z$ over $k$
such that $Z(k)$ is finite, nonempty, and computable.
\end{lemma}

\begin{proof}
First suppose that $\Char k=0$.
Let $E$ be the elliptic curve $X_1(11)$ over $k$.
By computing a Selmer group,
compute an integer $r$ strictly greater than
the rank of the finitely generated abelian group $E(k)$.
Let $Z=X_1(11^r)$ over $k$.
By \cite{Diamond-Shurman2005}*{Theorem~6.6.6},
the Jacobian $J_Z$ of $Z$ is isogenous to a product of $E^r$
with another abelian variety over $k$ (geometrically,
these $r$ copies of $E$ in $J_Z$
arise from the degeneracy maps $Z \to E$ indexed by $s \in \{1,\ldots,r\}$
that in moduli terms send $(A,P)$ to $(A/\langle 11^s P \rangle,11^{s-1} P)$
where $A$ is an elliptic curve and $P$ is a point on $A$ 
of exact order $11^r$).
So the Dem'janenko-Manin method~\cites{Demjanenko1966,Manin1969}
yields an upper bound on the height of points in $Z(k)$.
In particular, $Z(k)$ is finite and computable.
It is also nonempty, since the cusp $\infty$ on $X_1(11^r)$
is a rational point.

If $\Char k>0$,
let $Z$ be any nonisotrivial curve of genus greater than $1$
such that $Z(k)$ is nonempty:
for instance, let $a$ be a transcendental element of $k$,
and use the curve $C_a$ in the first paragraph of the proof
of Theorem~1.4 in \cite{Poonen-Pop2008}.
Then $Z(k)$ is finite by \cite{Samuel1966}*{Th\'eor\`eme~4},
and computable because of the height bound proved 
in \cite{Szpiro1981}*{\S8, Corollaire~2}.
\end{proof}

\begin{lemma}
\label{L:at least n}
Given a global field $k$ and a nonnegative integer $n$,
one can effectively construct a nice curve $Y$ over $k$
such that $Y(k)$ is finite, computable,
and of size at least $n$.
\end{lemma}

\begin{proof}
Construct $Z$ as in Lemma~\ref{L:at least 1}.
Let $\kappa(Z)$ denote the function field of $Z$.
Find a closed point $P \in Z - Z(k)$ whose residue field
is separable over $k$.

If $\Char k=0$,
the Riemann-Roch theorem, which can be made constructive,
together with a little linear algebra,
lets us find $f \in \kappa(Z)$ taking the value $1$ at each
point of $Z(k)$, and having a simple pole at $P$.
If $\Char k = p>2$,
instead 
find $t \in \kappa(Z)$ such that $t$ has a pole at $P$ and nowhere else,
and such that $t$ takes the value $1$ at each point of $Z(k)$;
then let $f=t+g^p$ for some $g \in \kappa(Z)$ such that $g$ has 
a pole at $P$ of odd order greater than the order of the pole of $t$ at $P$
and no other poles,
such that $g$ is zero at each point of $Z(k)$,
and such that $t+g^p$ is nonzero at each zero of $dt$;
this ensures that $f$ has an odd order pole at $P$ and no other poles,
and is $1$ at each point of $Z(k)$, and has only simple zeros
(since $f$ and $df=dt$ do not simultaneously vanish).
In either case, $f$ has an odd order pole at $P$,
so $\kappa(Z)(\sqrt{f})$ is ramified over $\kappa(Z)$ at $P$,
so the regular projective curve $Y$ with $\kappa(Y)=\kappa(Z)(\sqrt{f})$
is geometrically integral.
A local calculation shows that $Y$ is also smooth, so $Y$ is nice.
Equations for $Y$ can be computed by resolving singularities
of an initial birational model.
The points in $Z(k)$ split in $Y$, so $\#Y(k) = 2 \#Z(k)$,
and $Y(k)$ is computable.
Iterating this paragraph eventually produces a curve $Y$
with enough points.

If $\Char k = 2$,
use the same argument,
but instead adjoin to $\kappa(Z)$ 
a solution $\alpha$ to $\alpha^2-\alpha=f$,
where $f \in \kappa(Z)$ has a pole of high odd order at $P$,
no other poles, and a zero at each point of $Z(k)$.
\end{proof}

\begin{proof}[Proof of Theorem~\ref{T:main}]
Given $k$ and $n$,
apply Lemma~\ref{L:at least n}
to find $Y$ over $k$ with $Y(k)$ finite, computable,
and of size at least $n+4$.
Write $Y(k)=\{y_1,\ldots,y_m\}$.
Find a closed point $P \in Y-Y(k)$ with residue field separable over $k$.

Suppose that $\Char k \ne 2$.
Compute $a,b \in k^\times$
whose images in $k^{\times}/k^{\times 2}$ are $\F_2$-independent.
Let $S$ be the set of places $v \in k$ such that
$a$, $b$, and $ab$ are all nonsquares in $k_v$.
By Hensel's lemma, if $v \nmid 2,\infty$ and $v(a)=v(b)=0$,
then $v \notin S$.
So $S$ is finite and computable.
Let $w \in \Omega_k-S$.
Weak approximation~\cite{Artin-Whaples1945}*{Theorem~1}, 
whose proof is constructive, 
lets us find $c \in k^\times$ 
such that $c$ is a square in $k_v$ for all $v \in S$
and $w(c)$ is odd.
The purpose of $w$ is to ensure that $c$ is not a square in $k$.
Find $f \in \kappa(Y)^\times$ such that 
$f$ has an odd order pole at $P$
and a simple zero at each of $y_1,\ldots,y_n$,
and such that $f(y_{n+1})=a$, $f(y_{n+2})=b$, $f(y_{n+3})=ab$,
and $f(y_{n+4})=\cdots=f(y_m)=c$.
If $\Char k=p>2$, 
the same argument as in the proof of Lemma~\ref{L:at least n}
lets us arrange in addition that $f$ has no poles other than $P$,
and that all zeros of $f$ are simple.
Construct the nice curve $X$ whose function field is $\kappa(Y)(\sqrt{f})$.
Then $X \to Y$ maps $X(k)$ bijectively to $\{y_1,\ldots,y_n\}$,
so $X(k)$ is computable and of size $n$.
Also, for each $v \in \Omega_k$,
at least one of $a,b,ab,c$ is a square in $k_v$,
so $X(k_v) \ne \emptyset$.

If $\Char k = 2$,
use the same argument,
with the following modifications.
For any extension $L$ of $k$,
define the additive homomorphism $\wp\colon L \to L$ by $\wp(t)=t^2-t$.
Construct $a,b \in k$ such that the images of $a$ and $b$
in $k/\wp(k)$ are $\F_2$-independent.
Let $S$ be the set of places $v \in k$ such that
$a$, $b$, and $a+b$ are all outside $\wp(k_v)$.
As before, $S$ is finite and computable.
Choose $w \in \Omega_k-S$.
Use weak approximation to find $c \in k$ such that 
$c \in \wp(k_v)$ for all $v \in S$ but $c \notin \wp(k_w)$.
Find $f \in \kappa(Y)$
such that $f$ has a pole of high odd order at $P$,
a simple pole at $y_1,\ldots,y_n$,
and no other poles,
and such that $f(y_{n+1})=a$, $f(y_{n+2})=b$, $f(y_{n+3})=a+b$,
and $f(y_{n+4})=\cdots=f(y_m)=c$.
Construct the nice curve $X$ whose function field 
is obtained by adjoining to $\kappa(Y)$
a solution $\alpha$ to $\alpha^2-\alpha=f$.
\end{proof}

\section{Other constructions of curves violating the local-global principle}
\label{S:other constructions}

\subsection{Lefschetz pencils in a Ch\^atelet surface}

J.-L.~Colliot-Th\'el\`ene has suggested another approach
to constructing curves violating the local-global principle,
which we now sketch.
For any global field $k$, there exists a 
Ch\^atelet surface over $k$ violating the local-global principle:
see \cite{Poonen2009-chatelet}*{Proposition~5.1}
and \cite{Viray-preprint}*{Theorem~1.1}.
Let $V$ be such a surface.
Choose a projective embedding of $V$.
By \cite{Katz1973}*{Th\'eor\`eme~2.5}, 
after replacing $V$ by a $d$-uple embedding for some $d \ge 1$,
there is a Lefschetz pencil of hyperplane sections of $V$,
fitting together into a family $\tilde{V} \to \PP^1$,
where $\tilde{V}$ is the blowup of $V$ along
the intersection of $V$ with the axis of the pencil.
Since $\tilde{V} \to V$ is a birational morphism,
the Lang-Nishimura theorem
(see \cite{Nishimura1955}, \cite{Lang1954}*{Theorem~3},
and also \cite{CT-Coray-Sansuc1980}*{Lemme~3.1.1})
shows that $\tilde{V}$ has a $k$-point if and only if $V$ does,
and the same holds with $k$ replaced by any completion $k_v$.
By definition of Lefschetz pencil, each geometric fiber of the pencil 
is either an integral curve 
or a union of two nice curves intersecting
transversely in a single point.
By requiring $d \ge 3$ above, we can ensure that each
geometric fiber is also $2$-connected, which means that whenever it
decomposed as a sum $D_1+D_2$ of
two nonzero effective divisors, the intersection number $D_1.D_2$ is 
at least $2$
(the $2$-connectedness follows from~\cite{VandeVen1979}*{Theorem~I}; 
that paper is over $\C$, but the argument works in arbitrary characteristic).
This rules out the possibility of a geometric fiber with two components,
so every geometric fiber is integral.
The ``fibration method'' 
(see, e.g., \cite{CT-Sansuc-SD1987I}, \cite{Colliot-Thelene1998}*{2.1}, 
\cite{CT-Poonen2000}*{Lemma~3.1})
shows that there is a finite set of places $S$ such that 
for every place $v \notin S$ and every point $t \in \PP^1(k)$,
the fiber of $\tilde{V} \to \PP^1$ above $t$ has a $k_v$-point.
For $v \in S$, the set $\tilde{V}(k_v)$ is nonempty,
and its image in $\PP^1$ contains a nonempty open subset $U_v$ of $\PP^1(k_v)$.
By weak approximation, we can find $t \in \PP^1(k)$ such that $t \in U_v$
for all $v \in S$,
and such that the fiber of $\tilde{V} \to \PP^1$ above $t$ is smooth.
That fiber violates the local-global principle.

With a little work, this construction can be made effective.
On the other hand, this approach does not seem to let one construct
curves with a prescribed positive number of points.

\subsection{Atkin-Lehner twists of modular curves}

Theorem~1 of \cite{Clark2008} constructs a natural 
family of curves over $\Q$ violating the local-global principle:
namely, for any squarefree integer $N$ with $N>131$ and $N \ne 163$,
there is a positive-density set of primes $p$
such that the twist of $X_0(N)$ by the main Atkin-Lehner involution $w_N$
and the quadratic extension $\Q(\sqrt{p})/\Q$
violates the local-global principle over $\Q$.
See \cite{Clark2008} for details, and for a connection to the
inverse Galois problem.
The proof involves Faltings' theorem~\cite{Faltings1983},
so it does not yield an effective construction of a suitable pair $(N,p)$.

On the other hand, as P.~Clark explained to me,
a variant of this construction is effective,
and works over an arbitrary global field $k$.
His idea is to replace $X_0(N)$ above with 
a modular curve $X$ having both $\Gamma_0(N)$
and $\Gamma_1(M)$ level structures, for suitable $M$ and $N$
depending on $k$,
and to apply Merel's theorem (or a characteristic $p$ analogue) 
to $X_1(M)$ to control $X(k)$.
See~\cite{Clark-preprint} for details.

\begin{remark}
One can also find counterexamples to the local-global principle over $\Q$
among Atkin-Lehner {\em quotients} of Shimura curves:
see \cite{Rotger-Skorobogatov-Yafaev2005} and \cite{Parent-Yafaev2007}.
\end{remark}

\section*{Acknowledgements} 

I thank Pierre D\`ebes for the suggestion to use the Dem'janenko-Manin method.
I thank Pete L. Clark and Jean-Louis Colliot-Th\'el\`ene for sharing their
ideas sketched in Section~\ref{S:other constructions}.
I also thank Clark for a correction, and Izzet Coskun for suggesting
the reference \cite{VandeVen1979}.
Finally I thank the referee for a few suggestions.

\end{document}